\documentclass[pdflatex,sn-chicago]{sn-jnl}

\usepackage{graphicx}%
\usepackage{multirow}%
\usepackage{amsmath,amssymb,amsfonts}%
\usepackage{amsthm}%
\usepackage{mathrsfs}%
\usepackage[title]{appendix}%
\usepackage{xcolor}%
\usepackage{textcomp}%
\usepackage{manyfoot}%
\usepackage{booktabs}%
\usepackage{algorithm}%
\usepackage{algorithmicx}%
\usepackage{algpseudocode}%
\usepackage{listings}%
\usepackage{xfrac}

\theoremstyle{thmstyleone}%
\newtheorem{theorem}{Theorem}
\newtheorem{lemma}[theorem]{Lemma}
\newtheorem{corollary}[theorem]{Corollary}

\theoremstyle{thmstyletwo}%

\theoremstyle{thmstylethree}%
\newtheorem{definition}{Definition}%

\newcommand {\mm}[1] {\ifmmode{#1}\else{\mbox{\(#1\)}}\fi}

\newcommand {\ceiling}[1] {{\left\lceil #1 \right\rceil}}

\newcommand{\ignore}[1]{}

\newsavebox{\smallProofsym}                 
\savebox{\smallProofsym}                               %
{
\begin{picture}(6,6)
\put(0,0){\framebox(6,6){}}
\put(0,2){\framebox(4,4){}}
\end{picture} 
}

\newcommand{\Rspace}        {\mm{{\mathbb R}}}
\newcommand{\Zspace}        {\mm{{\mathbb Z}}}
\newcommand{\Cech}[1]       {\mm{{\check{C}}_{#1}}}

\newcommand{\Hgroup}[1]     {\mm{\sf H}_{#1}}

\newcommand{\Betti}[1]      {\mm{\beta}_{#1}}
\newcommand{\Hausdorff}[2]  {\mm{H}{({#1},{#2})}}
\newcommand{\Mesh}[1]       {\mm{{\rm mesh}{({#1})}}}

\newcommand{\maxBetti}[2]   {\mm{{M}_{{#1},{#2}}}}
\newcommand{\maxBettiP}[3]  {\mm{{M}_{{#1},{#2},{#3}}}}

\newcommand{\Glue}[2]       {\mm{{\gamma}|_{{#1}\sim{#2}}}}
\newcommand{\Sweep}[2]      {\mm{{\gamma}|^{{#1}\sim{#2}}}}
\newcommand{\Star}[2]       {\mm{\rm St}_{#1}{({#2})}}

\newcommand{\ee}            {\mm{\varepsilon}}
\newcommand{\Alpha}         {\mm{\mathrm A}}

\newcommand{\dist}[2]       {\mm{\|{#1}-{#2}\|}}

\newcommand{\card}[1]       {\mm{{\rm card\,}{#1}}}

\newcommand{\Skip}[1]       {}

\begin{document}

\title[Maximum Persistent Betti Numbers of \v{C}ech Complexes]{Maximum Persistent Betti Numbers of \v{C}ech Complexes}


\author[1]{\fnm{Herbert} \sur{Edelsbrunner}}\email{herbert.edelsbrunner@ist.ac.at}

\author[2]{\fnm{Matthew} \sur{Kahle}}\email{mkahle@math.osu.edu}

\author[3]{\fnm{Shu} \sur{Kanazawa}}\email{s.kanazawa@qmul.ac.uk}

\affil[1]{\orgname{ISTA (Institute of Science and Technology Austria)}, \orgaddress{\street{Am Campus 1}, \city{3400 Kloster\-neu\-burg}, \country{Austria}}}

\affil[2]{\orgdiv{Department of Mathematics}, \orgname{Ohio State University}, \orgaddress{\street{231 W.\ 18th Ave.}, \city{Columbus}, \postcode{43210}, \state{Ohio}, \country{United States}}}

\affil[3]{\orgdiv{School of Mathematical Sciences}, \orgname{Queen Mary University of London}, \orgaddress{\street{Mile End Road}, \city{London}, \postcode{E1 4NS}, \country{United Kingdom}}}


\abstract{
This note proves that only a linear number of holes in a \v{C}ech complex of $n$ points in $\mathbb{R}^d$ can persist over an interval of constant length.
Specifically, for any fixed dimension $p<d$ and fixed $\varepsilon > 0$, the number of $p$-dimensional holes in the \v{C}ech complex at radius $1$ that persist to radius $1 + \varepsilon$ is bounded above by a constant times $n$, where $n$ is the number of points.
The proof uses a packing argument supported by relating the \v{C}ech complexes with corresponding snap complexes over the cells in a partition of space.
The argument is self-contained and elementary, relying on geometric and combinatorial constructions rather than on the existing theory of sparse approximations or interleavings.
The bound also applies to Alpha complexes and Vietoris--Rips complexes.
While our result can be inferred from prior work on sparse filtrations, to our knowledge, no explicit statement or direct proof of this bound appears in the literature.
}




\maketitle

\section{Introduction}
\label{sec:1}

What is the maximum number of holes created by $n$ possibly overlapping closed unit balls in a Euclidean space, and how big can they be?
To move toward a more precise formulation of this question, let $p < d$ be two positive integers, and consider the asymptotic behavior of
\begin{align}
  \maxBetti{p}{d} (n) &= \max \left\{ \Betti{p} \left( \bigcup\nolimits_{x \in A} B(x, 1) \right) | A \subseteq \Rspace^d, \card{A} = n \right\}
  \label{eqn:Mpd}
\end{align}
as $n\to\infty$.
Here, $B(x, 1)$ is the closed ball of radius $1$ centered at $x$, and $\Betti{p}$ of the union of such balls is the $p$-th Betti number of this space.
The \emph{\v{C}ech complex} of $A$ for radius $r$, denoted by $\Cech{r} = \Cech{r}(A)$, is the simplicial complex whose vertices are the points in $A$, and whose simplices are the subsets of points such that the closed balls of radius $r$ centered at these points have a non-empty common intersection.
By the Nerve Theorem, the \v{C}ech complex for radius $r=1$ has the homotopy type of the union of the unit balls in \eqref{eqn:Mpd}.
Therefore, $\maxBetti{p}{d} (n)$ is also the maximum $p$-th Betti number of $\Cech{1}$, for any set of $n$ points in $\Rspace^d$.

\smallskip
Recently, Edelsbrunner and Pach~\citep{EdPa24} proved that $\maxBetti{p}{d} (n) = \Theta(n^{\min \{ p + 1, \ceiling{d/2} \}})$.
For example, $\maxBetti{1}{2} (n)$ grows linearly in $n$, but $\maxBetti{1}{3} (n)$ and $\maxBetti{2}{3} (n)$ grow quadratically in $n$.
The upper bound is easily derived from the Upper Bound Theorem for convex polytopes, see e.g.\ \citep{Zie95}, and their main contribution is the actual construction that proves that this upper bound is asymptotically tight.
However, most holes in their construction appear to be small: they vanish when the balls are slightly enlarged.
In the language of persistent homology, they have short lifetimes.
This motivates us to fix $\ee > 0$ and to look at
\begin{align}
  \maxBettiP{p}{d}{\ee} (n) &= \max \{\Betti{p}(\Cech{1}, \Cech{1 + \ee}) \mid  A \subseteq \Rspace^d, \card{A} = n\} ,
\end{align}
where $\Betti{p}(\Cech{1}, \Cech{1 + \ee})$ is the $p$-th persistent Betti number of the inclusion of $\Cech{1}$ in $\Cech{1 + \ee}$. 
In other words, this is the number of $p$-dimensional holes in $\Cech{1}$ that are still holes in $\Cech{1 + \ee}$.
Note that the scale need not be fixed at $r=1$: by rescaling the point configuration that attains the maximum, we may consider the persistent Betti number $\Betti{p}(\Cech{r}, \Cech{(1 + \ee)r})$ for any $r\ge0$.
When we fix the parameter $\ee > 0$, most of the holes in Edelsbrunner and Pach's construction are no longer counted since they do not persist even to radius $1+\ee$, so it is not surprising that the asymptotic behavior of $\maxBettiP{p}{d}{\ee} (n)$ is different from that of $\maxBetti{p}{d} (n)$, and we show that $\maxBettiP {p}{d}{\ee} (n)$ grows only linearly in $n$.
In words, the number of holes that cover a fixed interval of positive length in the \v{C}ech filtration is at most some constant, depending only on $\ee$ and the dimension $d$, times the number of points.
It may also be of interest to consider the case in which the constant $\ee>0$ depends on $n$.
For further details, see Section~\ref{sec:5}.

\smallskip
The \v{C}ech complex has the same homotopy type as the Alpha complex for the same points and the same radius \cite[Section~III.4]{EdHa10}, which is a subcomplex of the Delaunay mosaic of the points \citep{Del34}.
Hence, $\maxBetti{p}{d} (n)$ is also the maximum $p$-th Betti number of the Alpha complex of $n$ points in $\Rspace^d$, and $\maxBettiP{p}{d}{\ee} (n)$ is also the maximum $p$-th persistent Betti number of the Alpha complex for unit radius included in that of radius $1+\ee$.
Less obviously, our linear upper bound for $\maxBettiP{p}{d}{\ee} (n)$ extends to the inclusion of Vietoris--Rips complexes for unit radius in that of radius $1 + \ee$.
While we do not know of a direct connection, every step in our proof of the bound for \v{C}ech complexes extends to Vietoris--Rips complexes.
The resulting linear upper bound for the persistent Betti numbers should be compared to the bounds of Goff~\citep{Gof11} on the (non-persistent) Betti numbers of Vietoris--Rips complexes.
For $p=1$, he shows a linear upper bound in all dimensions, which is stronger than our linear upper bound for the persistent Betti numbers.
However, for $p > 1$, his bound is only $o (n^p)$, which is much weaker than our linear upper bound for the persistent Betti numbers.

\smallskip
\subsubsection*{Related work on sparse approximations of filtrations}
Although our initial motivation was to contrast the asymptotics of the maximum Betti number $\maxBetti{p}{d} (n)$ with those of the maximum persistent Betti number $\maxBettiP{p}{d}{\ee} (n)$ in the context mentioned above, we recognize that our main result is closely related to the literature on sparse approximations of filtrations.
In this line of work, a substantial body of literature shows that the \v{C}ech and Vietoris--Rips filtrations can be sparsified---significantly reducing their size---while still capturing similar topological information.

\smallskip
Persistent homology records how topological features (such as components, loops, and voids) appear and disappear as the scale parameter varies in a filtration; the birth-death intervals (bars) of these features are summarized in a barcode.
Many sparsification results are stated in terms of (multiplicative) interleavings.
Informally, and sufficient for our purposes, a $(1+\ee)$-interleaving between filtrations means that there are natural comparison maps consistently relating the complex at scale $\alpha$ in one filtration to the complex at scale $(1+\ee)\alpha$ in the other, and conversely.
In particular, each topological feature whose lifetime strictly exceeds a factor of $1+\ee$; that is: a bar that strictly covers an interval $[\alpha,(1+\ee)\alpha)$ for some $\alpha\ge0$ in one filtration has a corresponding feature in the other.

\smallskip
A seminal result is Sheehy's sparse Vietoris--Rips filtration~\citep{She13}.
For an $n$-point set in Euclidean space (more generally, for an $n$-point metric space of bounded doubling dimension), he constructs a linear-size ($O(n)$-size) filtration, based on a hierarchical net-tree, that is $(1+\ee)$-interleaved with the exact Vietoris--Rips filtration.
The hidden constant depends only on $\ee$ and the dimension of the point set.
The same philosophy has been developed for \v{C}ech filtrations.
Of particular relevance is Kerber--Sharathkumar's sparse \v{C}ech filtration~\citep{KerSha13}, which introduces the \emph{well-separated simplicial decomposition (WSSD)}, a higher-dimensional analogue of the well-separated pair decomposition.
Using a WSSD, they construct a linear-size filtration that is $(1+\ee)$-interleaved with the exact \v{C}ech filtration.
Further refinements and alternative viewpoints include Botnan--Spreemann's \v{C}ech approximation via sequences of good covers~\citep{BoSp15} and the geometric perspective of Cavanna--Jahanseir--Sheehy, who show that a sparse complex can be realized as a nerve in one dimension higher~\citep{CaJaShe15}.
A high-dimensional lattice-based approach by Choudhary--Kerber--Raghvendra yields polynomial-size approximations---independent of the dimensionality of the point set---together with a super-polynomial lower bound on the size required for a high-accuracy approximation, thereby demonstrating the inherent accuracy/size trade-off in high dimension~\citep{ChoKerRa19}.
For developments in the two-parameter setting, we refer the reader to~\citep{BuDoKer23,LeMc24} and the references therein.

\subsubsection*{How sparse approximation implies our linear bound on $\maxBettiP{p}{d}{\ee} (n)$}
We herein explain how earlier results on sparse approximation imply our linear bound on the maximum persistent Betti number $\maxBettiP{p}{d}{\ee} (n)$.
For a fixed $\ee>0$, suppose that in the \v{C}ech filtration $(\Cech{r})_{r\ge0}$ there exists a super-linear number of bars that strictly cover an interval $[\alpha,(1+\ee)\alpha)$ for some $\alpha\ge0$.
Any proper sparsification scheme described above produces a linear-size filtration that is (multiplicatively) $(1+\ee)$-interleaved with the exact \v{C}ech filtration and therefore still carries all such bars.
However, a linear-size filtration can support only a linear number of bars, because each bar requires at least one simplex to witness its creation or destruction.
This contradiction shows that, in the exact \v{C}ech filtration, the number of bars whose death-to-birth ratio is strictly greater than $1+\ee$ grows only linearly.
Note that this argument handles all such longer bars simultaneously, across all scales.
%

\subsubsection*{Our contributions}
As mentioned already, our main result is a linear upper bound on the maximum persistent Betti numbers over any fixed interval of positive length.
Precisely, for fixed $0\le s<t$, the $p$-th persistent Betti number $\Betti{p}(\Cech{s}, \Cech{t})$ is bounded above by some constant times the number of points in $\Rspace^d$, where the constant depends only on the ratio $t/s$ and the dimension $d$.
By rescaling the point configuration, we may therefore focus on $\Betti{p}(\Cech{1}, \Cech{1+\ee})$ with $\ee=t/s-1>0$.

\smallskip
Our method proceeds as follows: we mesh the Euclidean space into hypercubes of diameter at most $\ee$ and identify the points lying in each cell.
Applying this identification to the \v{C}ech complex $\Cech{1}$ produces a smaller simplicial complex $Q_1$, which we call a \emph{snap complex} (see Definition~\ref{dfn:snap_complex}).
The key observation can be summarized as follows: there is a triangle of simplicial maps formed by $\Cech{1}\to Q_1\to\Cech{1+\ee}$ and $\Cech{1}\to\Cech{1+\ee}$, which commutes on the homology level.
Consequently, the rank of the induced map $\Hgroup{p}(\Cech{1})\to\Hgroup{p}(\Cech{1+\ee})$---that is, $\Betti{p}(\Cech{1}, \Cech{1+\ee})$---is bounded above by $\dim(\Hgroup{p}(Q_1))$, which is in fact linear in the number of the input points by a packing argument (see the remarks in Section~\ref{sec:3} for details).

\smallskip
Our snap complex is akin in spirit to earlier work, particularly the constructions of~\citep{KerSha13} and~\citep{ChoKerRa19}.
Both also mesh Euclidean space into hypercubes or permutahedra, yet their foci differ from that of our snap complex argument.
\begin{enumerate}
  \item In~\citep{KerSha13}, the construction of a multiplicatively $(1+\ee)$-interleaved filtration exploits the well-separated simplicial decomposition (WSSD).
  Intuitively, a WSSD decomposes the points into small clusters whose diameters are tiny compared with the distances separating them and groups these clusters into tuples so that every \v{C}ech simplex is captured by at least one such tuple.
  The construction is more sophisticated and delicate because the clusters must be small relative to their mutual distances to guarantee a multiplicative approximation across all scales.
  Our snap complex approach, based on a uniform mesh size, avoids this technical step and proceeds directly to the bound on the persistent Betti number, since we essentially focus on a single scale.
  \item In~\citep{ChoKerRa19}, Euclidean space $\Rspace^d$ is meshed into permutahedral cells.
  For a given finite point set, each cell containing at least one input point is taken as a vertex of a new simplicial complex, and simplices are defined as the nerve of these occupied cells.
  Taking an arbitrary base scale $\beta>0$ and varying the mesh size over the geometric progression $\beta c^k$ with $c=4(d+1)^2$, they obtain a discrete tower of such complexes, connected by natural simplicial maps, which approximates the Vietoris--Rips filtration within an $O(d)$ multiplicative factor while achieving smaller size bounds than Sheehy's sparse Vietoris--Rips filtration.
  In contrast, our snap complex construction takes into account not only the mesh size but also the scale $r$, which allows us to build a sparse filtration $(Q_r)_{r\ge0}$ of snap complexes that approximates the classical filtrations with arbitrary accuracy---though only in an additive sense---by refining the mesh.
\end{enumerate}

\smallskip
While our main result---a linear upper bound on persistent Betti numbers over a fixed interval of positive length---can be deduced from earlier results on linear-sized approximations of filtrations as discussed above, our objective is of a different nature.
We present a direct and self-contained proof based on elementary geometric and combinatorial arguments.
Our approach circumvents the technically more involved steps of filtration approximations that deal with all scales simultaneously in a multiplicative sense, and addresses directly the bound on the maximum persistent Betti number at a single scale, thereby offering a more accessible and clear exposition of this fundamental fact.
Moreover, to the best of our knowledge, no explicit statement of this linear bound on persistent Betti numbers exists, despite its noteworthy contrast with the tight polynomial upper bounds on (non-persistent) Betti numbers of \v{C}ech complexes by Edelsbrunner and Pach~\citep{EdPa24}.
We believe it is worthwhile to clearly articulate and prove this fact, both for its intrinsic interest and for potential applications in stochastic topology (see Section~\ref{sec:5} for further discussion).

\subsubsection*{Outline}
The outline of this paper is as follows.
Section~\ref{sec:2} proves technical lemmas about how cycles are maintained if we glue vertices in a simplicial complex.
Section~\ref{sec:3} introduces the snap complex of a \v{C}ech complex and relates its Betti numbers to the persistent Betti numbers of the \v{C}ech complex.
Section~\ref{sec:4} combines these preparations to prove the linear upper bound on the persistent Betti number of \v{C}ech complexes.
Section~\ref{sec:5} concludes the paper.

\section{Gluing Vertices}
\label{sec:2}

Let $A \subseteq \Rspace^d$ be finite.
By definition, the \v{C}ech complex of $A$ for radius $r \geq 0$ consists of all subsets $B \subseteq A$ such that the closed balls of radius $r$ centered at the points in $B$ have a non-empty common intersection.
Writing $r(B)$ for the radius of the smallest enclosing sphere of $B$, we have $B \in \Cech{r} (A)$ iff $r(B) \leq r$.
Letting $B' \subseteq \Rspace^d$ be another finite set, the \emph{Hausdorff distance} between $B$ and $B'$ is
\begin{align}
  \Hausdorff{B}{B'} &= \max \left\{ \max_{b \in B} \min_{b' \in B'} \dist{b}{b'} , \max_{b' \in B'} \min_{b \in B} \dist{b'}{b}\right\} .
\end{align}
We show that the difference between $r(B)$ and $r(B')$ is at most the Hausdorff distance between the two sets.
\begin{lemma}
  \label{lem:smallest_enclosing_spheres}
  $| r(B) - r(B') | \leq \Hausdorff{B}{B'}$.
\end{lemma}
\begin{proof}
  Let $x$ be the center of the smallest enclosing sphere of $B$,
  whose radius is $r(B)$.
  By definition of Hausdorff distance, $B'$ is contained in the union of balls of radius $\ee = \Hausdorff{B}{B'}$ centered at the points in $B$.
  Hence, the sphere with center $x$ and radius $r(B) + \ee$ encloses $B'$, so $r(B') \leq r(B) + \ee$.
  Symmetrically, $r(B) \leq r(B') + \ee$, so $r(B') - \ee \leq r(B) \leq r(B') + \ee$, which implies the claim.
\end{proof}

We employ Lemma~\ref{lem:smallest_enclosing_spheres} to reduce the size of a cycle.
To explain how, we briefly review standard terminology from homology.
For simplicity, we work over the coefficient field $\Zspace / 2 \Zspace$, though the results extend to any field (see the remark immediately after Corollary~\ref{cor:persistent_Betti_numbers} for a more algebraic approach).
A \emph{$p$-chain}, $\gamma$, is a collection of $p$-simplices, and the \emph{vertices} of $\gamma$ are the vertices of its $p$-simplices.
To \emph{add} two $p$-chains means taking their symmetric difference: $\gamma + \gamma' = \gamma \oplus \gamma'$.
The \emph{boundary} of a $p$-simplex are its $(p-1)$-dimensional faces, and the \emph{boundary} of $\gamma$ is the sum of the boundaries of its $p$-simplices.
A \emph{$p$-cycle} is a $p$-chain with empty boundary, denoted by $\partial \gamma = 0$.
Two $p$-cycles are \emph{homologous}, denoted by $\gamma \sim \gamma'$, if there is a $(p+1)$-chain, $\Gamma$, that satisfies $\partial \Gamma = \gamma + \gamma'$.
In this case, we say $\Gamma$ is a \emph{filling} of $\gamma + \gamma'$.
Finally, $\gamma$ is \emph{trivial} if $\gamma \sim 0$.
\begin{definition}
  \label{dfn:gluing_vertices}
  Let $\gamma$ be a $p$-cycle and $x \neq y$ vertices of $\gamma$.
  To \emph{glue} $x$ and $y$, we substitute a new vertex, $z$, for $x$ and $y$ in all $p$-simplices, and write $\Glue{x}{y}$ for the resulting $p$-chain.
\end{definition}
After the substitution, a $p$-simplex that contains both, $x$ and $y$, is a $(p-1)$-simplex and thus implicitly removed from the $p$-cycle by the gluing operation.
%
\begin{figure}[hbt]
    \centering \vspace{0.0in}
    \resizebox{!}{1.3in}{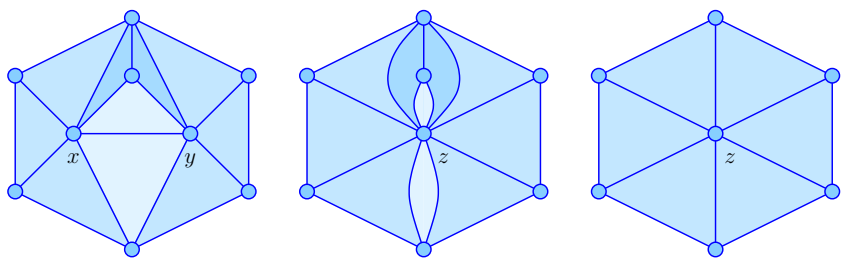}
    \vspace{-0.0in}
    \caption{\footnotesize \emph{Left:} a portion of a $2$-cycle in which $x$ and $y$ belong to two triangles that share an edge different from the edge connecting $x$ to $y$.
    For better visibility, we shade the triangles in the stars of $x$ and $y$ \emph{dark}, \emph{light}, and \emph{in between} depending on whether they share such an edge, they belong to both stars, and neither, respectively.
    \emph{Middle:} the contraction of the edge connecting $x$ to $y$ produces two bi-gons and two triangles that share the same three vertices.
    \emph{Right:} these bi-gons and duplicate triangles are removed.}
    \label{fig:gluing}
\end{figure}
As illustrated in Figure~\ref{fig:gluing}, it is possible that after substituting $z$ for $x$ and $y$, we have two $p$-simplices with the same $p+1$ vertices, which, by definition, are two copies of the same $p$-simplex.
By the logic of modulo-$2$ arithmetic, the two copies cancel each other.
Substituting a new vertex $z$ for two distinct vertices $x$ and $y$ in a simplicial complex defines a simplicial map, as every simplex is mapped to a simplex in the new complex with $x$ and $y$ replaced by $z$.
Since the $p$-chain $\Glue{x}{y}$ is nothing but the image of the $p$-cycle $\gamma$ under this simplicial map, $\Glue{x}{y}$ is also a $p$-cycle.
We state the fact without proof, as it is elementary.
\begin{lemma}
  \label{lem:gluing_maintains_cycles}
  Let $\gamma$ be a $p$-cycle and $x \neq y$ two vertices of $\gamma$.
  Then $\Glue{x}{y}$ is a $p$-cycle.
\end{lemma}
Writing $\Star{\gamma}{x}$ and $\Star{\gamma}{y}$ for the $p$-simplices in $\gamma$ that contains $x$ and $y$, respectively, we define
\[
\Sweep{x}{y} = z \cdot [\Star{\gamma}{x} \cup \Star{\gamma}{y}],
\]
where the multiplication means taking the cone over the $p$-simplices in the union of two stars; in more combinatorial terms, this corresponds to adding the new vertex $z$ to each set of $p+1$ vertices.
Intuitively, $\Sweep{x}{y}$ is the $(p+1)$-chain that is swept out by the $p$-simplices in $\Star{\gamma}{x} \cup \Star{\gamma}{y}$ as we move $x$ and $y$ on straight lines to $z$.
In the configuration displayed in Figure~\ref{fig:gluing}, this corresponds to the $3$-chain whose boundary consists of the $10$ and $6$ triangles in the left and right drawings, respectively.
It is not difficult to see that the boundary of this $(p+1)$-chain is the sum of the two $p$-cycles.
\begin{lemma}
  \label{lem:sweep_is_filling}
  Let $\gamma$ be a $p$-cycle, $x \neq y$ two vertices of $\gamma$, $\gamma' = \Glue{x}{y}$, and $\Gamma = \Sweep{x}{y}$.
  Then $\partial \Gamma = \gamma + \gamma'$.
\end{lemma}
\begin{proof}
  The $p$-simplices in $\gamma$ that neither belong to $\Star{\gamma}{x}$ nor to $\Star{\gamma}{y}$ also belong to $\gamma'$.
  Indeed, they exhaust $\gamma' \setminus \Star{\gamma'}{z}$, which implies $\gamma + \gamma' = \Star{\gamma}{x} \cup \Star{\gamma}{y} \cup \Star{\gamma'}{z}$.
  By construction of $\Gamma = \Sweep{x}{y}$, and the assumption that $\gamma$ is a $p$-cycle, the right-hand side of this equation is the boundary of $\Gamma$.
\end{proof}

\section{Snap Complex}
\label{sec:3}

For values $0 \leq s \leq t$, we write $\Cech{s} = \Cech{s}(A)$ for the \v{C}ech complex, $\Betti{p}(\Cech{s})$ for the rank of its $p$-th (reduced) homology group, $\Hgroup{p}(\Cech{s})$, and $\Betti{p} (\Cech{s}, \Cech{t})$ for the rank of the image of $\Hgroup{p}(\Cech{s})$ in $\Hgroup{p}(\Cech{t})$ induced by the inclusion $\Cech{s} \subseteq \Cech{t}$.
\begin{definition}
  \label{dfn:snap_complex}
  Let $\Psi$ be a partition of $\Rspace^d$ into cells, and call the supremum diameter of the sets in $\Psi$ the \emph{mesh} of the partition, denoted by $\Mesh{\Psi}$.
  The \emph{snap complex}, $Q_s$, of the \v{C}ech complex $\Cech{s}$ along $\Psi$ is defined as follows:
  the vertices of $Q_s$ are the cells that contain at least one point of $A$, and a set $\{\psi_0, \psi_1, \ldots, \psi_p\} \subseteq \Psi$ of distinct cells is a $p$-simplex in $Q_s$ iff there exist vertices $x_i \in \psi_i$ for $0 \leq i \leq p$ such that $\{x_0, x_1, \ldots, x_p\} \in \Cech{s}$.
\end{definition}
Note that mapping each $p$-simplex $\{x_0, x_1, \ldots, x_p\}\in \Cech{1}$ with $x_i \in \psi_i\in\Psi$~($0 \leq i \leq p$) to a simplex $\{\psi_0, \psi_1, \ldots, \psi_p\}\in Q_1$ defines a surjective simplicial map $q \colon \Cech{1} \to Q_1$.
Letting $\gamma$ be a $p$-cycle in $\Cech{1}$, we call $\alpha = q(\gamma)$ its \emph{image}, and $\gamma$ a \emph{preimage $p$-cycle} of $\alpha$.
Given two preimage $p$-cycles of $\alpha$, they may or may not be homologous in $\Cech{1}$.
The next lemma, however, guarantees that they are homologous in $\Cech{1+\Mesh{\Psi}}$.
\begin{lemma}
  \label{lem:homologous_preimages}
  Let $\ee = \Mesh{\Psi}$, $\alpha$ a $p$-cycle in $Q_1$, and $\gamma, \gamma_3$ two preimage $p$-cycles of $\alpha$ in $\Cech{1}$.
  Then $\gamma \sim \gamma_3$ in $\Cech{1+\ee}$.
\end{lemma}
\begin{proof}
  We construct a sequence of homologous $p$-cycles that interpolates between $\gamma$ and $\gamma_3$ in $\Cech{1+\ee}$.
  There is an initial sequence interpolating between $\gamma$ and $\gamma_1$, a middle sequence interpolating between $\gamma_1$ and $\gamma_2$, and a terminal sequence interpolating between $\gamma_2$ and $\gamma_3$.

  \smallskip
  The initial sequence reduces the number of vertices until the preimage $p$-cycle contains at most one vertex per cell in $\Psi$.
  Indeed, if $\gamma$ has vertices $x \neq y$ in the same cell, then we can glue $x$ and $y$, as described in Section~\ref{sec:2}, which produces a new $p$-chain, $\gamma' = \Glue{x}{y}$.
  This operation introduces a new vertex, $z$.
  We are free to choose its location, and to facilitate the repetition of this argument, we choose it where $y$ used to be.
  By Lemma~\ref{lem:gluing_maintains_cycles}, $\gamma'$ is a $p$-cycle with one fewer vertices than $\gamma$.
  Let $\Gamma = \Sweep{x}{y}$ be the $(p+1)$-chain from Lemma~\ref{lem:sweep_is_filling}.
  If all simplices in $\Gamma$ belong to $\Cech{1+\ee}$, then $\gamma$ and $\gamma'$ are homologous in $\Cech{1+\ee}$, so assume that at least one simplex $\upsilon \in \Gamma$ does not belong to $\Cech{1+\ee}$.
  Then $r(\upsilon) > 1 + \ee$.
  By construction of $\Gamma$, and the choice of $z$'s location in the cell that also contains $x$ and $y$, there is a $p$-simplex $\tau$ in $\gamma$ that has a vertex in every cell that contains a vertex of $\upsilon$.
  Since the points in the same cell have distance at most $\ee$ from each other, Lemma~\ref{lem:smallest_enclosing_spheres} implies $r(\tau) > r(\upsilon) - \ee > 1$.
  But then $\tau$ is not in $\Cech{1}$ and neither is $\gamma$, which contradicts the assumptions.
  This implies $\gamma \sim \gamma'$ in $\Cech{1+\ee}$.
  By repeating the argument,  all $p$-cycles in the initial sequence are homologous in $\Cech{1+\ee}$ and, in particular, $\gamma \sim \gamma_1$ in $\Cech{1+\ee}$.

  \smallskip
  The terminal sequence of $p$-cycles reduces the number of vertices of $\gamma_3$ if read from the end forward.
  Appealing again to Lemmas~\ref{lem:smallest_enclosing_spheres}, \ref{lem:gluing_maintains_cycles}, and \ref{lem:sweep_is_filling}, all $p$-cycles in the terminal sequence are homologous in $\Cech{1+\ee}$ and, in particular, $\gamma_2 \sim \gamma_3$ in $\Cech{1+\ee}$.

  \smallskip
  Since $\gamma$ and $\gamma_3$ are preimage $p$-cycles of the same $p$-cycle, $\alpha$, their vertices lie in the same cells of $\Psi$, and so do the vertices of $\gamma_1$ and $\gamma_2$.
  The middle sequence interpolates between the latter two $p$-cycles by changing one vertex at a time to a possibly different point in the same cell.
  Appealing to Lemmas~\ref{lem:smallest_enclosing_spheres}, \ref{lem:gluing_maintains_cycles}, and \ref{lem:sweep_is_filling}, the $p$-cycles in the middle sequence are again homologous in $\Cech{1+\ee}$ and, in particular, $\gamma_1 \sim \gamma_2$ in $\Cech{1+\ee}$.
  But now we have $\gamma \sim \gamma_1 \sim \gamma_2 \sim \gamma_3$ in $\Cech{1+\ee}$ and therefore $\gamma \sim \gamma_3$ in $\Cech{1+\ee}$, as claimed.
\end{proof}

\noindent \emph{Remark on Vietoris--Rips complexes.}
Lemma~\ref{lem:homologous_preimages} generalizes to Vietoris--Rips complexes.
Indeed, its proof generalizes provided we adapt Lemma~\ref{lem:smallest_enclosing_spheres}, which in its current formulation is specific to \v{C}ech complexes.
For the Vietoris--Rips complexes, we read the radius $r(B)$ as half of the maximum distance between any two vertices in $B$.
Then it is still true that the difference in radii is bounded from above by the Hausdorff distance between the two sets of vertices.
In other words, Lemma~\ref{lem:smallest_enclosing_spheres} also applies to Vietoris--Rips complexes.

\smallskip
By Lemma~\ref{lem:homologous_preimages}, any two preimage cycles of a cycle in $Q_1$ that already exist in $\Cech{1}$ are homologous in $\Cech{1+\ee}$.
We can therefore bound the persistent Betti numbers by the Betti number of the snap complex.
\begin{corollary}
  \label{cor:persistent_Betti_numbers}
  With $\ee = \Mesh{\Psi}$, we have $\Betti{p} (\Cech{1}, \Cech{1+\ee}) \leq \Betti{p}(Q_1)$ for every $p$.
\end{corollary}
\begin{proof}
  Recall that $\Betti{p} (\Cech{1}, \Cech{1+\ee} )$ is the rank of the persistent homology group that captures all $p$-cycles in $\Cech{1+\ee}$ that already exist in $\Cech{1}$.
  It suffices to prove that the images of non-trivial $p$-cycles in $\Cech{1+\ee}$ that already exist in $\Cech{1}$ are non-trivial $p$-cycles in $Q_1$.
  To derive a contradiction, let $\gamma$ be a $p$-cycle in $\Cech{1}$ that is non-trivial in $\Cech{1+\ee}$, and assume that $\alpha = q(\gamma)$ is trivial in $Q_1$.
  Hence, there exists a $(p+1)$-chain, $\Alpha$, in $Q_1$ with $\partial \Alpha = \alpha$, and because $q$ is surjective, there also exists a $(p+1)$-chain $\Gamma$ in $\Cech{1}$ with $q (\Gamma) = \Alpha$.
  Noting that $\gamma$ and $\partial \Gamma$ are both $p$-cycles in $\Cech{1}$ whose images under $q$ are equal to $\alpha$, we obtain $\gamma \sim \partial \Gamma \sim 0$ in $\Cech{1+\ee}$ by Lemma~\ref{lem:homologous_preimages}.
  This contradicts that $\gamma$ is non-trivial in $\Cech{1+\ee}$ and implies the claimed inequality.
\end{proof}
\noindent \emph{Remark on the proof of Corollary~\ref{cor:persistent_Betti_numbers}.}
Alternatively, Corollary~\ref{cor:persistent_Betti_numbers} can be deduced using a more algebraic argument based on contiguous simplicial maps.
Starting with a simplex in $\Cech{1}$, one may either include it directly into $\Cech{1+\ee}$ or first map it to $Q_1$ and then further to $\Cech{1+\ee}$, which is possible by Lemma~\ref{lem:smallest_enclosing_spheres}.
These two simplicial maps are contiguous, meaning that for every simplex in $\Cech{1}$, the union of its images under the two maps is also a simplex in $\Cech{1+\ee}$.
It is a standard fact that contiguous simplicial maps induce the same homomorphism on homology groups.
Consequently, the triangle of simplicial maps formed by $\Cech{1}\to Q_1\to\Cech{1+\ee}$ and $\Cech{1}\to\Cech{1+\ee}$ commutes on the homology level, which immediately yields Corollary~\ref{cor:persistent_Betti_numbers}.
Although this proof is arguably cleaner, we have chosen to present the current proof to make the geometric ideas more transparent.

\medskip \noindent \emph{Remark on the right-hand side of the inequality.}
It is easy to see that the bound in Corollary~\ref{cor:persistent_Betti_numbers} is not tight.
Take for example three points equally spaced on a circle of radius strictly between $1$ and $1 + \ee$.
Then $\Cech{1}$ has a $1$-cycle of three edges, while in $\Cech{1+\ee}$ this cycle is filled by the triangle.
Hence, $\Betti{1} (\Cech{1}, \Cech{1+\ee}) = 0$, which is strictly smaller than $\Betti{1} (Q_1) = 1$.
\begin{figure}[hbt]
    \centering \vspace{0.0in}
    \resizebox{!}{1.7in}{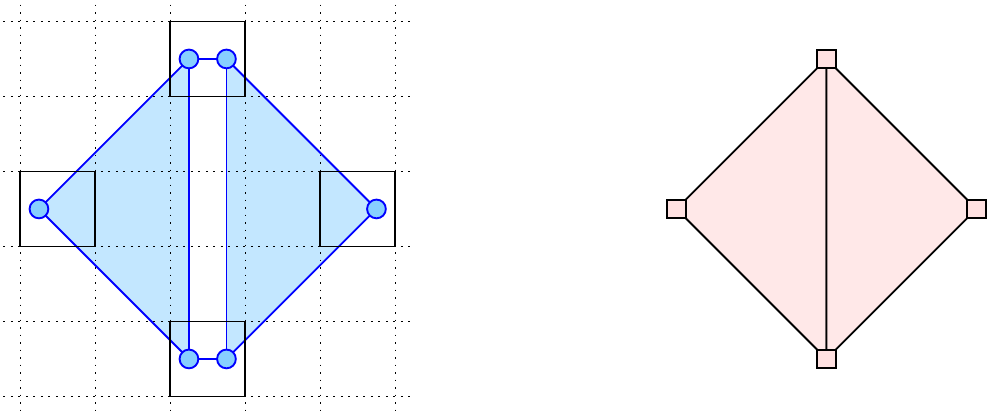}
    \vspace{-0.0in}
    \caption{\footnotesize \emph{Left:} the \v{C}ech complex for six points inside four cells in the partition of the plane.
    Assuming the two triangles are isosceles, right-angled, and have smallest enclosing circles of  radius $1+\ee$, the narrow rectangle between the two triangles has a smallest enclosing circle with radius strictly larger than $1+\ee$.
    Hence, the boundary of the convex hexagon that passes through the six points is a non-trivial $1$-cycle in $\Cech{1+\ee}$, and for $\ee \leq \sqrt{2}-1$, it already exists in $\Cech{1}$.
    \emph{Right:} the image of the hexagon is a quadrangle in the snap complex.  Its boundary is a trivial $1$-cycle in $Q_{1+\ee}$ because the rectangle collapses to a single edge shared by the images of the two triangles.}
    \label{fig:2triangles}
\end{figure}
Note also that we cannot replace the upper bound $\Betti{p}(Q_1)$ in Corollary~\ref{cor:persistent_Betti_numbers} with $\Betti{p}(Q_{1+\ee})$ in general.
The reason is that the image of a homologically non-trivial cycle in $\Cech{1+\ee}$ may be homologically trivial in $Q_{1+\ee}$, even if it already exists in $\Cech{1}$; see Figure~\ref{fig:2triangles} for an example.
Indeed, we have $1 = \Betti{1} (\Cech{1}, \Cech{1+\ee}) > \Betti{1} (Q_{1+\ee}) = 0$ in this example.

\medskip \noindent \emph{Remark on Vietoris--Rips complexes.}
As mentioned earlier, Lemma~\ref{lem:homologous_preimages} generalizes to Vietoris--Rips complexes.
With this, the proof of Corollary~\ref{cor:persistent_Betti_numbers} generalizes to Vietoris--Rips complexes.

\section{The Upper Bound}
\label{sec:4}

We use a packing argument together with Corollary~\ref{cor:persistent_Betti_numbers} to prove that for every $\ee > 0$, the number of homology classes born before or at $1$ and dying after $1+\ee$ in the \v{C}ech filtration of $n$ points in $\Rspace^d$ is bounded from above by a constant times $n$.
This constant depends on $\ee$ and $d$ but not on $n$.
\begin{theorem}
  \label{thm:holes_in_Cech_complex}
  For every $\ee > 0$, there exists $c = c(\ee, d)$ such that $\Betti{p}(\Cech{1}, \Cech{1+\ee}) \leq c \cdot n$.
\end{theorem}
\begin{proof}
  We partition $\Rspace^d$ into translates of $[0, \sfrac{\ee}{\sqrt{d}})^d$.
  The diameter of every cell is $\ee$, so we call this partition $\Psi$ and apply Corollary~\ref{cor:persistent_Betti_numbers}.
  Fixing $\psi_0 = [0, \sfrac{\ee}{\sqrt{d}})^d \in \Psi$, the cells that are connected to $\psi_0$ by an edge in $Q_1$ must contain a point at distance at most $2$ from a point in $\psi_0$.
  Therefore, such cells lie inside the hypercube $[-2-\sfrac{\ee}{\sqrt{d}}, 2+\sfrac{2\ee}{\sqrt{d}})^d$.
  Its volume is $(4+\sfrac{3\ee}{\sqrt{d}})^d$.
  Comparing this with the volume of a single cell, which is $(\sfrac{\ee}{\sqrt{d}})^d$, the number of such cells is at most
  \begin{align}
    C(\ee,d) &= \frac{(4+3\ee/\sqrt{d})^d}{(\ee/\sqrt{d})^d}
           = \left( 3 + \frac{4\sqrt{d}}{\ee} \right)^d.
    \label{eqn:C}
  \end{align}
  To span a $p$-simplex in $Q_1$, we pick the fixed cell and add $p$ from the at most $C = C(\ee,d)$ cells within the mentioned distance.
  We thus have at most $\binom{C}{p} n$ $p$-simplices in $Q_1$, where $n$ is the number of ways we can fix the first cell.
  The number of $p$-simplices is an upper bound on the $p$-th Betti number.
  By Corollary~\ref{cor:persistent_Betti_numbers}, the same upper bound applies to the number of $p$-cycles born before or at $1$ and dying after $1+\ee$.
  We have non-zero persistent Betti numbers only for $p < d$, so $c = \binom{C}{p} < 2^{C}$
  is a constant for which the claimed inequality holds.
\end{proof}

\noindent \emph{Remark on Vietoris--Rips complexes.}
Since Corollary~\ref{cor:persistent_Betti_numbers} generalizes to Vietoris--Rips complexes, so does Theorem~\ref{thm:holes_in_Cech_complex}.

\section{Discussion}
\label{sec:5}

The main result of this paper is a linear upper bound on the number of holes in the \v{C}ech complex of $n$ points in $\Rspace^d$ that persist from radius $r = 1$ to $r = 1+\ee$, in which $\ee$ is a fixed constant strictly larger than $0$.
The upper bound generalizes to the Alpha complex and the Vietoris--Rips complex and thus holds for three of the classic types of complexes used in topological data analysis \citep{Car09,EdHa10}.
The work reported in this short note raises a number of questions, and we mention two.

\smallskip
The first natural question is how small we can make $\ee > 0$ in our linear upper bound on the persistent Betti numbers when we think of $\ee$ as a function of $n$ that tends to $0$, rather than a fixed positive constant.
The construction in \citep{EdPa24} shows that the maximum $p$-th Betti number of a \v{C}ech complex with $n$ vertices in $\Rspace^d$ is $\Theta (n^m)$, with $m = \min \{p+1, \ceiling{\sfrac{d}{2}} \}$.
A detailed look at the analysis shows that for even $d \geq 4$, the persistence of the counted cycles is proportional to $\sfrac{1}{n^2}$ (and smaller for odd $d \geq 3$), in which we simplify by assuming that $d$ is a constant.
In other words, in even dimensions the lower bound extends to the persistent Betti numbers, $\Betti{p} (\Cech{1}, \Cech{1+\ee})$, provided $\ee = o(\sfrac{1}{n^2})$.
The upper bound for the constant of proportionality in Theorem~\ref{thm:holes_in_Cech_complex} depends on $\ee$ and $d$ in a way that suggests it grossly overestimates the number of holes that persist.
Indeed, for $\ee = \sfrac{1}{n^2}$, the constant $C$ in the proof grows as $\Theta (n^{2d})$ and the constant $c$ grows as $\Theta (n^{2d(d-1)})$.
Can this upper bound be improved to showing that the polynomially many holes in the lower bound construction of Edelsbrunner and Pach~\citep{EdPa24} are asymptotically as persistent as possible?
Alternatively, can this lower bound construction be improved to increase the persistence of the holes?
As a related result to this question, it is shown in~\citep{ChoKerRa19} that if $\ee<O(1/\log^{1+c}(n))$ for some $c\in(0,1)$, then the \v{C}ech filtration of an $n$-point set must have $n^{\Omega(\log\log n)}$ bars whose death-to-birth ratio is strictly greater than $1+\ee$.
However, this result holds under the assumption that the dimension satisfies $d=\Theta(\log n)$.

\smallskip
The second question is motivated by the last author's quest to prove the large deviation principle for persistent Betti numbers and persistence diagrams of random \v{C}ech filtrations (cf.~\citep{HO23,KHMT24}).
Let $A$ be a finite set of points in a large $d$-dimensional cube, partitioned into points $L$ and $R$ to the left and right of a vertical hyperplane, respectively.
Our goal is to approximate the persistent Betti numbers of the \v{C}ech filtration of $A$ by the sum of those for the \v{C}ech filtrations of $L$ and $R$.
Writing $a_p = \Betti{p} (\Cech{1}(A), \Cech{1+\ee}(A))$, $\ell_p = (\Cech{1}(L), \Cech{1+\ee}(L))$, and $r_p = \Betti{p} (\Cech{1}(R), \Cech{1+\ee}(R))$, we desire a bound on the absolute difference between $a_p$ and $\ell_p + r_p$.
Letting $M \subseteq A$ contain the points at distance at most $2(1+\ee)$ from the hyperplane, the vertices of every simplex in $\Cech{1+\ee}(A) \setminus (\Cech{1+\ee}(L) \cup \Cech{1+\ee}(R))$ must belong to $M$, so it is natural to estimate the absolute difference in terms of $M$:
assuming $\ee > 0$ is a fixed constant and $0 \leq p<d$, is it true that there exists a constant such that
\begin{align}
  | a_p -  (\ell_p + r_p) |
  \leq {\rm const} \cdot \card{M} ?
  \label{eqn:diff_PB}
\end{align}
In other words, is the absolute difference between these persistent Betti numbers bounded from above by a constant times the number of points in the narrow strip next to the hyperplane?
The absolute difference is of course bounded by the number of simplices spanned by these points (see, e.g.,~\cite[Lemma~2.11]{HST18} and~\cite[Proposition~16]{KHMT24}), but this only implies that the left-hand side of~\eqref{eqn:diff_PB} is bounded from above by $f_p(\Cech{1}(M)) + f_{p+1}(\Cech{1+\ee}(M))$, and thus by $2 (\card{M})^{p+2}$.
Therefore, the significance of the inequality~\eqref{eqn:diff_PB} lies in the linear bound with respect to the number of points in the narrow strip.

\subsection*{Funding}
The three authors are supported by 
the Wittgenstein Prize, Austrian Science Fund (FWF), grant no.\ Z 342-N31, by the DFG Collaborative Research Center TRR 109, Austrian Science Fund (FWF), grant no.\ {I 02979-N35, 
the U.S.\ National Science Foundation (NSF-DMS), grant no.\ 2005630,
and a JSPS Grant-in-Aid for Transformative Research Areas (A) (22H05107, Y.H.), EPSRC Research Grant EP/Y008642/1, respectively.}


\subsection*{Acknowledgements}
The authors would like to thank Michael Lesnick and Primoz Skraba for their helpful comments regarding sparse approximations of filtrations.


\bibliography{sn-bibliography}

\end{document}